\documentclass[12pt]{amsart}
\usepackage{a4wide}
\usepackage{verbatim}

\newtheorem{theorem}{Theorem}

\newtheorem{definition}{Definition}
\newtheorem{proposition}{Proposition}
\newtheorem{example}{Example}
\newtheorem{corollary}{Corollary}
\newtheorem{remark}{Remark}

\def\REQ#1{{\rm (\ref{#1})}}
\def\d{{\rm d}}
\def\D{\displaystyle}

\def\cG{\mathcal{G}}

\def\cH{\mathcal{H}}
\def\cO{\mathcal{O}}

\def\cS{\mathcal{S}}

\def\e{\text{\rm e}}
\def\d{\text{\rm d}}
\def\i{\text{\rm i}}

\def\P{\mathbb{P}}
\def\bR{\mathbb{R}}

\numberwithin{equation}{section}

\begin{document}

\author{Zdzislaw Brze{\'z}niak}
\address{Department of Mathematics, University of York,  York, YO10 5DD, UK}
\email{zdzislaw.brzezniak@york.ac.uk}

\author{Ben Goldys}
\address{School of Mathematics and Statistics, The University of Sydney, Sydney 20006, Australia}
\email{Beniamin.Goldys@sydney.edu.au}

\author{Szymon Peszat}
\address{Institute of Mathematics, Polish Academy of Sciences, \'Sw. Tomasza 30/7, 31-027 Krak\'ow, Poland}
\email{napeszat@cyf-kr.edu.pl}

\author{Francesco Russo}
\address{Ecole Nationale Sup\'erieure des Techniques Avanc\'ees,
ENSTA-ParisTech, Unit\'e de Ma\-th\'e\-matiques appliqu\'ees,
828, boulevard des Maréchaux
F-91120 Palaiseau (France)}
\email{Francesco.Russo@ensta-paristech.fr}

\title{Second Order PDEs with Dirichlet White Noise Boundary Conditions}
\thanks{Szymon Peszat's research  was supported by Polish Ministry of Science and Higher Education  Grant N N201 419039 ``Stochastic Equations in Infinite Dimensional Spaces''.}
\thanks{ Ben Goldys research was supported by the ARC grant DP120101886}
\thanks{Francesco Russo was partially supported
by the ANR Project MASTERIE 2010 BLAN 0121 01.}

\begin{abstract}
In this paper we study the Poisson and heat equations on bounded and unbounded domains with smooth boundary with random Dirichlet boundary conditions. The main novelty of this work is a convenient framework for the analysis of such equations excited by the white in time and/or space noise on the boundary. Our approach allows us to show the existence and uniqueness of weak solutions in the space of distributions. Then we prove that the solutions can be identified as smooth functions inside the domain, and finally the rate of their blow up at the boundary is estimated. A large class of noises including Wiener and fractional Wiener space time white noise,  homogeneous noise and L\'evy noise is considered.
\end{abstract}
\maketitle

\medskip

{\bf MSC Classification 2010: }
60H15; 35J25; 35K10; 35K51; 60G20.

\section{Introduction}
The aim of this work is to develop a framework for systematic analysis of elliptic and parabolic boundary value problems with random boundary conditions including space-time white noise.

We focus in this paper on the Dirichlet problem only because the Dirichlet boundary conditions  pose the greatest challenge. However, let us  emphasise that our approach can be used to study both the elliptic and the parabolic equations with various boundary conditions in any dimension.

Let $\mathcal O$ be a (possibly unbounded) domain in $\mathbb{R}^d$. If $d>1$ then we will assume that the boundary $\partial {\mathcal O}$ is of class $C^\infty$. In this paper we are concerned with the existence and regularity of solutions to the following Poisson and heat equations
\begin{equation}\label{1E1}
\left\{\begin{array}{rcl}
\Delta u &=&\lambda u\qquad \mathrm{on}\ \mathcal O,\\
 u &=& \gamma \qquad \ \  \mathrm{on}\  \partial\mathcal O,
\end{array}\right.
\end{equation}
\begin{equation}\label{1E2}
\left\{
\begin{array}{rcl}
\D \frac {\partial u}{\partial t} &=& \Delta u \qquad \mbox{on}\
(0,+\infty)\times {\cO},\\
\D u &=&  \D\frac {\d \xi}{\d t}\qquad \,  \mbox{on}\
(0,+\infty)\times \partial {\cO},\\
u(0,\cdot)&=& u_0, \qquad \mbox{on}\ {\cO},
\end{array}
\right.
\end{equation}
where in  \eqref{1E1},   $\lambda>0$, if the domain is unbounded and $\lambda\ge 0$ in the case of a bounded $\cO$. We note that in the parabolic case the boundary condition     takes non-standard form $\frac{\d\xi}{\d t}$ in order to conform to notation used in the It\^o calculus.

If $\gamma$ and $\xi$ are deterministic and $\xi$ is time-differentiable then elliptic and parabolic problems \eqref{1E1} and \eqref{1E2}  are special cases of the general theory of boundary value problems that have been   thoroughly studied, see for example a basic monograph by  Lions and Magenes \cite{Lions-Magenes}, where the theory of weak solutions is developed. A semigroup approach to the inhomogeneous boundary value problems was initiated by Balakrishnan \cite{Balakrishnan}, for further developments see e.g \cite{Delfour, Lasiecka}.

Extension of the deterministic theory to the case of random boundary conditions is a natural next step motivated by various problems of mathematical physics, see for example \cite{Bricmont}.  It turns out that, except certain degenerate problems, a solution to equation \eqref{1E2} cannot have Markov  property if the noise $\xi$ is regular in time. Therefore there is a demand for the theory of equations with boundary data that are distributions in time. In fact, we will demonstrate in Section \ref{S4} that if $\xi$ is a Wiener or, more generally, a L\'evy process,  then the solution has the Markov property in an appropriate state
space. Another generalisation of the classical boundary conditions  leads to  the space-time white noise on the boundary. A typical example of such a situation arises, when $\mathcal O=(0,+\infty)\times\mathbb R\subset\mathbb R^{2}$ and formally
$$
\xi(t,x)=\sum_{k=1}^\infty W_k(t)e_k(x)\quad x\in\partial\mathcal O\equiv \mathbb{R},
$$
where $\left\{e_k\colon  k\ge 1\right\}$ is an orthonormal basis in $L^2(\mathbb{R},\mathcal{B}(\mathbb{R}),\d x)$ and $\left\{W_k\colon  k\ge 1\right\}$ is a collection of independent, one-dimensional Brownian Motions. Let us recall that such a process $\xi$, known as a cylindrical or white in space Wiener process,  is in fact a generalised random process; that is a distribution-valued process. It will be shown in Section \ref{S6} that equation \eqref{1E2} with such a noise still has a function-valued solution in $\mathcal O$. Other classes of generalised  noises will be studied as well, see Section \ref{S2}.

The existence of Markovian solutions to equation \eqref{1E2} with either Dirichlet or Neumann boundary conditions has been studied in few papers only.  It was
 proved in \cite{DaPrato-Zabczyk,DZ2,Freidlin-Sowers, Maslowski, Sowers} that under some assumptions equation \eqref{1E2} with the  Neumann boundary conditions has an $L^2(\cO)$-valued solutions.  For the wave equations with noise entering through the Neumann boundary condition see  e.g. \cite{DL1, L}.


It turns out that the Dirichlet problem is more difficult. In fact, Da Prato and Zabczyk demonstrated in \cite{DaPrato-Zabczyk}, see also \cite{DZ2, Peszat-Zabczyk},  that even in dimension one the solution to \REQ{1E2}, with $\xi$ being a Wiener process, takes values in the Sobolev space $H^{-\alpha}$ only if $\alpha>\frac{1}{2}$, hence is not function-valued. For this reason equation \REQ{1E2} was usually studied with more regular boundary noise see \cite{CS1, CS2, DS}. However, Al\`os and Bonaccorsi observed in \cite{AB1, AB2} that solution to (\ref{1E2}) considered on $\mathcal O=(0,+\infty)$ can be well defined in the space $L^2(0,+\infty ;\rho(x)\d x)$ with $\rho(x)=\min\left(1,x^{1+\theta}\right)\e^{-x}$ and $\theta>0$. This idea was further developed in \cite{GF}, where it was shown that for $\theta\in(0,1)$ the solution defined in \cite{AB1} can be identified as a mild solution in the spirit of \cite{DZ1}.

Finally, it is shown in \cite{Peszat-Zabczyk}  that if
$$
\xi(t)(x)= \int_0^t \int_{\partial \cO} \int_S \phi(x,y,z)\widehat \pi(\d r,\d y,\d z),
$$
where $\phi\colon [0,+\infty)\times \partial \cO \times S\mapsto \mathbb{R}$ is sufficiently regular function and $\widehat \pi$ is a Poisson random measure, then $\REQ{1E2}$ defines Markov family on $L^p(\cO)$ for some   $p\in [1,2)$.

Another class of problems, the so-called random dynamic boundary value problems was studied in \cite{CS1}. We note that in this case, the boundary condition is regular in time and therefore is not related to the main difficulty we deal with in this work. In this paper we are concerned with elliptic or parabolic problems. For hyperbolic problems with noise driven from the boundary see e.g. \cite{BP2, L}.

We will describe now in more detail the content of the paper. In section \ref{S2} we present some preliminary definitions and results needed in the following sections.
The concepts of weak solutions to equations \eqref{1E1} and \eqref{1E2} are introduced in Section \ref{S3} . Then in Section \ref{S4} we show that solution $u$ to problem \REQ{1E2} is  given by
\begin{equation}\label{1E3}
u(t)= S(t) u_0 + \int_0^t S(t-r) (\lambda-\Delta_\tau )D \d \xi(r),
\end{equation}
where $S$  is  the semigroup generated in $\mathcal H^\star$ (see Section \ref{S4} for definition) by the Laplace operator $\Delta_\tau$ with homogeneous boundary conditions, and $D$ is  the Dirichlet map which is defined as a weak solution to \REQ{1E1}. Therefore, $u$  is the  mild solution to the evolution equation
\begin{equation}\label{1E4}
\d u = \Delta _\tau u\d t + (\lambda-\Delta_\tau)D \d \xi.
\end{equation}
Then we will show (see Theorems \ref{4T1} and \ref{4T2}), that  for any  $\cS^\prime(\partial \mathcal{O})$-valued random variable $\gamma$ and any $\cS^\prime(\partial \cO)$-valued stochastic process $\xi$, there exist unique  weak solutions to equations $\REQ{1E1}$ and $\REQ{1E2}$ respectively. Extending the idea of Al\`os and Bonaccorsi  we show, see Section \ref{S6},  that under some natural assumptions on $\gamma$ and $\xi$, the solutions to \REQ{1E1} and \REQ{1E2} are $C^\infty$ in space (and in time in the case of the parabolic problem) inside the domain. Finally, see Section \ref{S7} we calculate in particular cases the rate of blow up of the solutions near the boundary $\partial \mathcal{O}$.

\section{Generalised random elements and processes}\label{S2}
We shall denote by $(\cdot,\cdot)$ the duality form on $\mathcal{S}^\prime(\mathbb{R}^d)\times \mathcal{S}(\mathbb{R}^d)$.  Let $K$ be a closed subset of $\mathbb{R}^d$.  We denote by $\cS^\prime(K)$ the class of all tempered distributions on $\mathbb{R}^d$ such that
$$
(\gamma, \psi) =0, \qquad \forall\, \psi\in \mathcal{S}(\mathbb{R}^d)\colon \psi=0 \ \text{on}\ K.
$$
We shall denote by $\mathcal{S}(K)$  the space of the restrictions of the test functions from $\mathcal{S}(\mathbb{R}^d)$ to the set $K$.
Obviously, $\cS^\prime(K)$ is a closed subspace of $\mathcal{S}^\prime(\mathbb{R}^d)$. Moreover the map
\[ \Lambda: \cS^\prime(K) \times \cS(K) \ni (\gamma, \psi_{|K})\mapsto (\gamma,\psi)\in\mathbb{R}\]
is well defined and bilinear. In what follows we will usually denote $\Lambda(\gamma, \psi_{|K})$ by $(\gamma, \psi_{|K})$.
   On $\mathcal{S}^\prime(K)$ we consider the topology inherited from $\mathcal{S}^\prime(\mathbb{R}^d)$. Then the Borel $\sigma$-field  $\mathcal{B}(\mathcal{S}^\prime(K))$  is  generated by the family of functions $(\cdot,\psi)$, $\psi \in \mathcal{S}(K)$. Note however that we do not claim that the space $\mathcal{S}^\prime(K)$ is a dual of $\mathcal{S}(K)$.

Let $(\Omega,\mathfrak{F},\mathbb{P})$ be a probability space. In this paper we assume that in the elliptic problem $\gamma$ is a measurable mapping from $\Omega$ to $\mathcal{S}^\prime(\partial \mathcal{O})$, and in the parabolic problem $\xi$ is an $\mathcal{S}^\prime(\partial \mathcal{O})$-valued process with c\`adl\`ag trajectories.

Let $L_0(\Omega,\mathfrak{F},\mathbb{P})$ be the space of all real-valued random variables on $(\Omega,\mathfrak{F},\mathbb{P})$ equipped with the topology of convergence in probability. Assume that $E$ is a locally convex  topological vector space and let $E^\prime$ be the topological dual space, see for instance \cite{Yos_1995}. Let  $(\cdot,\cdot)$ denote the corresponding duality bilinear form on $E^\prime\times E$.  A linear continuous mapping $\gamma \colon E\mapsto L_0(\Omega,\mathfrak{F},\mathbb{P})$ is called \emph{generalised linear random element on $E$}. For the reader convenience we recall representation theorem for generalised linear random elements. For its proof see e.g. \cite{I}.
\begin{theorem}
Assume that $E$ is nuclear and  $\gamma$ is  a generalised linear random
 element on $E$. Then there exists a measurable
 $\tilde {\gamma} \colon \Omega\mapsto E^\prime$ such that for any $\psi \in E$,
 \begin{equation}\label{eqn-representation1}
 \gamma(\psi)=\left(\tilde{\gamma},\psi\right), \;\;\mathbb{P}-\mbox{a.s.}
\end{equation}
\end{theorem}
A typical application of the above Theorem is for $E=\mathcal{S}(\mathbb{R}^d)$.
However, because of separability of $\mathcal{S}(\mathbb{R}^d)$,  the following generalisation of it is also true.\\
If $\gamma$ is  a generalised linear random element on $\mathcal{S}(\mathbb{R}^d)$
such that for all $\psi\in \mathcal{S}(\mathbb{R}^d)\colon \psi=0$ on  $K$,
$\psi\in \mathcal{S}(K)$
\begin{equation}\label{eqn-representation2}
 \gamma (\psi )=0 \;\; \mathbb{P}-\mbox{a.s.}
 \end{equation}
then there exist a
measurable  map $\tilde {\gamma} \colon \Omega\mapsto \mathcal{S}^\prime(K)$ such that for any $\psi\in \mathcal{S}(K)$,
\begin{equation}\label{eqn-representation3}
 \gamma(\psi)=\left(\tilde{\gamma},\psi\right), \;\;\mathbb{P}-\mbox{a.s.}
 \end{equation}

From now on we identify a generalised linear random element $\gamma$ on $E$ with the corresponding $E^\prime$-valued random variable $\tilde \gamma$. We will also identify
a generalised linear random element $\gamma$ on $\mathcal{S}(\mathbb{R}^d)$
  satisfying condition \eqref{eqn-representation2}  with the corresponding $\mathcal{S}^\prime(K)$-valued random variable $\tilde \gamma$.

Let us present now  the most typical examples of $\mathcal{S}^\prime(\partial \mathcal{O})$-valued random variables $\gamma$ ad processes $\xi$.   Below, $\nu$ is  a Borel tempered measure on $\partial \mathcal{O}$; that is  $\nu\in \mathcal{S}^\prime(\partial \mathcal{O})$, and  $(e_k)$ is   an orthonormal basis of $L^2(\partial \mathcal{O},\mathcal{B}(\partial \mathcal{O}),\nu)$.

\begin{example}\label{2Ex1}
{\rm  Let $(\gamma_k)$  be a sequence of independent identically distributed zero-mean Gaussian   real-valued random  variables defined on $(\Omega, \mathfrak{F},\mathbb{P})$. Then the formula
$$
\gamma(\psi)= \sum_{k} \gamma_k (e_k\nu, \psi), \;\; \psi \in \mathcal{S}(\partial \mathcal{O}),
$$
where the series converges in $L^2(\Omega,\mathfrak{F},\mathbb{P})$, defines
a generalised random element on $\mathcal{S}(\partial \mathcal{O})$ satisfying condition \eqref{eqn-representation2} and  thus inducing  an
$\mathcal{S}^\prime(\partial \mathcal{O})$-valued random variable $\tilde \gamma$. As we explained earlier, we will identify $\gamma$ with $\tilde \gamma$ and use the earlier notation to denote the latter object. In particular, $\mathbb{P}$-a.s.,
$$
(\gamma, \psi)= \sum_{k} \gamma_k (e_k\nu, \psi), \;\; \psi \in \mathcal{S}(\partial \mathcal{O}).
$$
  This $\gamma$  will be called a  \emph{cylindrical Gaussian random variable} in $L^2(\partial \mathcal{O},\mathcal{B}(\partial \mathcal{O}),\nu)$ or \emph{Gaussian white noise on $\partial \mathcal{O}$} with \emph{intensity measure $\nu$}. }
\end{example}

\begin{example}\label{2Ex2}
{\rm Let $(W_k^H)$  be a sequence of independent fractional Brownian Motions   with a Hurst index $H\in (0,1)$.  Then the formula
$$
(\xi , \psi)(t)= \sum_{k} W_k^H(t) (e_k\nu, \psi), \qquad t\ge 0,\;\; \psi \in \mathcal{S}(\partial \mathcal{O}),
$$
where the series converges in $L^2(\Omega,\mathfrak{F},\mathbb{P})$, defines  an $\mathcal{S}^\prime(\partial \mathcal{O})$-valued random process.  We call $\xi$ the  \emph{cylindrical fractional Wiener  process in $L^2(\partial \mathcal{O},\mathcal{B}(\partial \mathcal{O}),\nu)$  with Hurst index $H$}. If $H=1/2$, then $\xi$ is called simply  a \emph{cylindrical  Wiener  process in $L^2(\partial \mathcal{O},\mathcal{B}(\partial \mathcal{O}),\nu)$}. }
\end{example}

\begin{example}\label{2Ex3}
{\rm Let $(\gamma_k)$ be a sequence of independent $\mathcal{N}(0,1)$-random variables, $(W^H_k)$ be a sequence of independent fractional Brownian Motions with Hurst index $H$ and let $(\nu_k)$ be a sequence of signed measures on $(\partial \mathcal{O},\mathcal{B}(\partial \mathcal{O}))$. Assume that
$$
\sum_{k} \|\nu_k\|_{\rm Var}^2<\infty,
$$
 where $\|\nu_k\|_{\rm Var}$ stands for the total variation of $\nu_k$. Then
 $$
 \gamma = \sum_{k} \gamma_k\nu_k,
 $$
 and
 $$
 W^H(t) =\sum_k W_k^H(t)\nu_k
 $$
 are $\mathcal{S}^\prime(\partial \mathcal {O})$ random variable and process.}
\end{example}

Let $-\infty <a<b<+\infty$. Assume now that $\mathcal{O}= (a,+\infty)\times \mathbb{R}^l$ or $\mathcal{O}= (a,b)\times \mathbb{R}^l$
for some positive integer $l$. Then $\partial \mathcal{O}$ can be identified with  $\mathbb{R}^l$ and $\mathbb{R}^{2l}$, respectively.  In these cases important classes of random processes on the boundary are provided by the so-called \emph{homogeneous Wiener processes}, see \cite{BP1, PZ1, Peszat-Zabczyk, P, PT}.
\begin{example}\label{2Ex4}
{\rm  Let  $\mathcal{W}$ be  a Wiener  process  taking values in $\mathcal{S}^\prime(\mathbb{R}^m)$, i.e.  $\mathcal{W}$ is  Gaussian process with continuous trajectories in $\mathcal{S}^\prime(\mathbb{R}^m)$ such that  for each  $\psi\in\mathcal{S}(\mathbb{R}^m)$, $t\mapsto  (\mathcal{W}(t),\psi)$ is a one  dimensional Wiener process. Let  $Q\colon\mathcal{S}(\mathbb{R}^m)\times \mathcal{S}(\mathbb{R}^m)\rightarrow\mathbb{R}$  be a  bilinear continuous  symmetric positive definite form defined by
$$
\mathbb{E}\,  ( \mathcal{W}(t), \psi)(\mathcal{W}(s), \varphi ) =  t\land s\, Q (\psi , \varphi )
\qquad
\text{for $(t,\psi ),(s,\varphi )\in [0,+\infty)\times \mathcal{S}(\mathbb{R}^m)$, $t\ge 0$.}
$$
We say that   $\mathcal W$ is \emph{spatially homogeneous} if  for each fixed $t\ge 0$ the law of $\mathcal{W}(t)$ is invariant with  respect to all translations $\tau ^\prime_h\colon \mathcal{S}^\prime(\mathbb{R}^m)\rightarrow \mathcal{S}^\prime(\mathbb{R}^m)$,  $h\in \mathbb{R}^m$, where  $\tau _h\colon \mathcal{S}(\mathbb{R}^m)\rightarrow \mathcal{S}(\mathbb{R}^m)$,  $\tau _h \psi (\cdot) = \psi (\cdot +h)$ for $\psi \in \mathcal{S}(\mathbb{R}^m)$.  Thus we assume that
$$
\mathbb{P}  \left( \mathcal{W}(t)\in \mathcal{X}\right) = \mathbb{P} \left( \mathcal{W}(t)\in  (\tau ^\prime_h)^{-1}(\mathcal{X}) \right)  \qquad  \text{for $h\in \mathbb{R}^m$, $\mathcal{X}\in \mathcal{B}(\mathcal{S}^\prime(\mathbb{R}^m))$.}
$$
The  property above holds if and only if  $Q$ is translation invariant, that is
\[Q (\psi ,\varphi)=  Q(\tau _h\psi ,\tau _h\varphi ),\quad\mathrm{for\,\, all}\quad \psi ,\varphi \in \mathcal{S}(\mathbb{R}^m),\,\,  h\in \mathbb{R}^m,\]
 see \cite{PZ1}. Next, see \cite{PZ1},  $Q$ is translation invariant if and only it is of the form  $Q (\psi ,\varphi )=\langle \Gamma , \psi \ast\varphi_{\rm {(s)}}\rangle$, where $\Gamma \in \mathcal{S}^\prime(\mathbb{R}^m)$ is the  Fourier transform of a positive symmetric tempered measure $\mu $ on  $\mathbb{R}^m$, and $\varphi _{\rm {(s)}}(x)=\overline{{\varphi (-x)}}$ for  $x\in {\mathbb{R}}^d$ and a complex-valued function $\varphi$.  We call $Q$ the \emph{covariance form}, and  $\mu$ the \emph{spectral  measure} of $\mathcal{W}$. Let us list the main properties of $\mathcal{W}$.
\begin{itemize}
\item
For each $\psi \in \mathcal{S}(\mathbb{R}^m)$,  $\{( {\mathcal{W}}(t),\psi) \}_{t\in[0,\infty)}$  is a real valued Wiener process.
\item
There exists $\Gamma\in\mathcal{S}^\prime(\mathbb{R}^m)$ such that for all  $\psi,\varphi\in{\mathcal{S}}$ one has
$$
Q (\psi ,\varphi ):= \mathbb{E} \,
( \mathcal{W}(1),\psi ) ( \mathcal{W}(1),\varphi) =
( \Gamma , \psi \ast\varphi_{\rm {(s)}}).
$$
\item
$\Gamma $ is the Fourier transform of a positive and symmetric Borel measure $\mu$ on $\mathbb{R}^m$, satisfying  $\int _{\mathbb{R}^d}(1+|x|)^r\d \mu (x)<\infty$ for a certain  $r<0$.
\item It was shown in \cite{PZ1} that $\mathcal{W}$ can be represented as a sum
$$
\mathcal{W}(t)=\sum_{k} W_k(t)\mathcal{F}(e_k\nu),
$$
where $(W_k)$ are independent standard real-valued Wiener processes, $\mathcal F$ stands for the Fourier transform, and $(e_k)$ is an orthonormal basis of the space  $L_{\rm{(s)}}^{2}(\mathbb{R}^d,\mathcal{B}(\mathbb{R}^d),  \mu)$ being  the closed subspace of $L^2(\Bbb R^d,\mathcal{B}(\mathbb{R}^d),\d \mu;\Bbb C)$ over the real field, consisting of all functions $u$ such that $u_{\rm{(s)}} = u$.
\end{itemize}

\smallskip \noindent
{\bf Random field.} Suppose that the spectral  measure $\mu$ of ${\mathcal W}$ is finite, and consequently  $\Gamma = \mathcal F( \mu) $ is a uniformly continuous bounded function.  Then there exists a Gaussian random field on $[0,+\infty )\times \mathbb{R}^m$,  which we also denote by  $\mathcal{W}$, such that:
\begin{itemize}
\item  The mapping $(t,x)\mapsto \mathcal{W}(t,x)$ is continuous with  respect to the first and measurable with respect to the both variables  $\mathbb{P} $-almost surely.
\item For each $x$, $\{\mathcal W(t,x)\}_{t\in [0,+\infty)}$  is a one dimensional Wiener process.
\item
$\mathbb{E} \, \mathcal{W}(t,x)\mathcal{W}(s,y) = t\land  s\, \Gamma (x-y)$  for $t,s\in [0,\infty)$ and $x,y \in \mathbb{R}^m$.
\item
$( \mathcal{W} (t), \psi) =  \int _{\mathbb{R}^m} \psi (x) \mathcal{W}(t,x)\d x$  for $\psi \in \mathcal{S}(\mathbb{R}^m)$.
\end{itemize}

In particular, if $\Gamma (x) = \e^{-|x|^\alpha }$,  where $\alpha \in (0,2]$, then we obtain important examples of random fields known as symmetric  $\alpha$-stable distributions . For $\alpha =1$ and $\alpha =2$ the densities  of the spectral measures are given by the formulas  $c_1(1+ |x|^2)^{-\frac{m+1}{2}}$ and $c_2\e ^{-|x|^2}$, where $c_1$  and $c_2$ are appropriate constants.

\smallskip\noindent
{\bf White noise.} If $Q(\psi ,\varphi )=\left(\psi ,\varphi  \right)$,  then $\Gamma $ is equal to the Dirac $\delta _0$-function, its  spectral density $\frac{\d \mu}{\d x}$ is the constant function $(2\pi )^{-\frac{m}{2}}$ and
$\mathcal{W}$ is a cylindrical Wiener process on $L^2(\mathbb{R}^m,\mathcal{B}(\mathbb{R}^m),\d x)$, see Example \ref{2Ex2},  and $\dot {\mathcal{W}}= \frac{\partial\mathcal{W}} {\partial t}$ is a white noise on $L^2([0,+\infty)\times \mathbb{R}^m)$. If  $B(t,x)$, $t\ge 0$ and $x\in \mathbb{R}^m$, is a Brownian sheet on  $[0,+\infty)\times {\mathbb{R}}^m$, see \cite{W}, then $\mathcal W$ can be defined by the formula,
$$
{\mathcal{W}}(t) = \frac{\partial ^d B(t)}{\partial x_1\ldots \partial x_d},
\qquad t\ge 0.
$$
}
\end{example}

\begin{example}\label{2Ex5}
{\rm Assume that $\pi$ and $\Pi$  are  Poisson random measures on $\partial \mathcal{O}$ and $[0,+\infty)\times \partial \mathcal{O}$ with intensity measures  $\mu$ and $\d t\nu$, respectively.  Then
$$
\left(\gamma,\psi\right) = \int_{\partial \mathcal{O}} \psi(y) \pi(\d y), \qquad \psi \in \mathcal{S}(\partial \mathcal{O}),
$$
and
$$
\left(\xi,\psi\right)(t) = \int_0^t \int_{\partial \mathcal{O}} \psi(y) \Pi(\d t, \d y), \qquad t\ge 0, \ \psi \in \mathcal{S}(\partial \mathcal{O}),
$$
define $\mathcal{S}^\prime(\partial \mathcal{O})$-valued random element and process, respectively. More generally, assume that $\rho\colon \partial \mathcal{O}\mapsto \mathbb{R}$ is a measurable function of a  polynomial growth. Then
$$
\left(\gamma_\rho,\psi\right) = \int_{\partial \mathcal{O}} \psi(y) \rho(y)\pi(\d y), \qquad \psi \in \mathcal{S}(\partial \mathcal{O}),
$$
and
$$
\left(\xi_\rho ,\psi\right)(t) = \int_0^t \int_{\partial \mathcal{O}} \psi(y)\rho(y) \Pi(\d t, \d y), \qquad t\ge 0, \ \psi \in \mathcal{S}(\partial \mathcal{O}),
$$
defines $\mathcal{S}^\prime(\partial \mathcal{O})$-valued random variable and process, respectively.
 }
\end{example}

\section{Weak solutions}\label{S3}

Let $\mathbf{n}$ denote the unit normal to the boundary and pointing inwards. Taking into account the Green formula (see e.g. \cite{Lions-Magenes}) we arrive at the following definitions of the weak solution to \REQ{1E1}.

\begin{definition}\label{3D1}
{\rm Let $\gamma$ be an ${\mathcal{S}}^\prime(\partial {\mathcal{O}})$-valued random variable. We call an ${\mathcal{S}}^\prime(\overline{\mathcal{O}})$-valued random variable $u$ a \emph{weak solution} to $\REQ{1E1}$ if
$$
(u, \Delta \psi) +\left(\gamma,  \frac {\partial \psi}{\partial \mathbf{n}}\bigg \vert_{\partial \mathcal{O}} \right)=\lambda( u,\psi),\quad \mathbb{P}-a.s.\quad \forall\, \psi\in \cS(\overline{\cO})\colon\  \psi= 0 \ \mbox{on} \ \partial {\cO}.
$$
}
\end{definition}

Note that the heat semigroup generated by the Laplace operator with Dirichlet boundary condition is not strongly continuous on $\mathcal{S}(\overline{\mathcal{O}})$. This is the main reason why in order to consider the parabolic problem we  have to introduce a certain scale of Hilbert spaces. Namely, let  $\vartheta(x)={(1+|x|^2)^{-1}}$ for $x\in \mathbb{R}^d$. Then, see e.g. \cite{PZ1, BP1, DZ2}, the heat semigroup $T(t)$, $t\ge 0$, generated by the Laplace operator $\Delta_\tau$ with homogeneous Dirichlet boundary conditions is  analytic (but not symmetric) on each   space $L^2_{\vartheta^r}:= L^2(\mathcal{O}, \mathcal{B}(\mathcal{O}),(\vartheta(x))^r\d x)$, $r\ge 0$.  Let $\lambda >0$ and let $H_r$, $r\ge 0$,  be the domain of $(\lambda -\Delta_\tau)^{r/2}$ considered on $L^2_{\vartheta^r}$. The space $H_r$ is equipped with the scalar product induced from $L^2_{\vartheta^r}$ by  $(\lambda -\Delta_\tau)^{r/2}$;
$$
\langle u,v\rangle_{H_r}= \langle ((\lambda -\Delta_\tau)^{r/2}u,(\lambda -\Delta_\tau)^{r/2}v\rangle_{L^2_{\vartheta^r}}.
$$
Let $H_{-r}$ be the dual space, where the duality map is given by
$$
H_r\hookrightarrow L^2_{\vartheta^r}\equiv \left(L^2_{\vartheta^r}\right)^*\equiv L^2(\mathcal{O}, \mathcal{B}(\mathcal{O}),(\vartheta(x))^{-r}dx)\hookrightarrow H_{-r}.
$$
Finally, let
$$
\mathcal{H}:= \bigcap_{r\ge 0} H_r,
$$
with the Fr\'echet topology, and let $\mathcal{H}^\prime$ be the topological dual. It is known that
$$
\mathcal{H}^\prime:=  \bigcup_{r\le 0} H_{-r}.
$$
Note that since $\mathcal{H}$ is a nuclear space, $\mathcal{H}^\prime$ is a co-nuclear space. Next,  the restriction of the heat semigroup $T$ to $H_r$, $r\ge 0 $, as well as to $\mathcal{H}$, is a $C_0$-semigroup. Moreover, $T$ can be extended to a $C_0$-semigroup on any $H_{-r}$-space as well as to $\mathcal{H}^\prime$.  Note that the generator of $T$ considered either on $\mathcal{H}$ or on $\mathcal{H}^\prime$ is a continuous linear operator. With a bit of abuse of notation, we will denote by $\Delta_\tau$ the generator of $T$ regardless the space on which $T$ is considered.

\begin{remark}\label{3R1}
{\rm Note that $\mathcal{H}$ is a subspace of $\mathcal{S}(\overline{\mathcal{O}})$. Therefore any distribution from  $\mathcal{S}^\prime(\overline{\mathcal{O}})$ can be treated as an element of $\mathcal{H}^\prime$.  Note however that $\mathcal{H}$ is not a dense subspace of  $\mathcal{S}(\overline{\mathcal{O}})$. In fact it is a closed proper subspace. Therefore it can happen that a non-zero $\xi\in \mathcal{S}^\prime(\overline{\mathcal{O}})$ vanishes on $\mathcal{H}$.}
\end{remark}

\begin{definition}\label{3D2}
{\rm Let $\xi$ be an $\mathcal{S}^\prime(\partial \mathcal{O})$-valued random process. We say that an  $\mathcal{H}^\prime$-valued random process   $u$ is a \emph{weak solution} to $\REQ{1E2}$ starting from  $u_0\in \mathcal{H}^\prime$, if
$$
(u(t),\psi) =(u_0,\psi) +\int_0^t (u(r), \Delta \psi)\d r + \left(\xi(t),\frac {\partial \psi}{\partial \mathbf{n}}\bigg \vert_{\partial \mathcal{O}}\right)- \left(\xi(0),\frac {\partial \psi}{\partial \mathbf{n}}\bigg \vert_{\partial \mathcal{O}}\right),
$$
$\P$-a.s. for all for all $t>0$, and  $\psi\in \mathcal{H}$.}
\end{definition}

For the completeness of the presentation we present the following result on the uniqueness of solutions.

\begin{proposition}\label{3P1}
For any constant $\lambda\ge  0$, and  any    $\cS^\prime(\partial \cO)$-valued random variable $\gamma$, the elliptic problem $\REQ{1E1}$ has  at most one solution.
\end{proposition}
\begin{proof} Clearly it is enough to show that solution to the problem with homogeneous boundary conditions vanishes, and we take this for granted in the case of $\cO=\bR^d$. In other words we will use the fact that if $u$ is an  $\mathcal{S}^\prime(\mathbb{R}^d)$-valued random variable  such that
$$
\left(u,\Delta \psi\right)= \lambda \left(u,\psi\right),\qquad \forall\, \psi\in \mathcal{S}(\mathbb{R}^d),
$$
then $u\equiv 0$.

Assume that $u$ solves \REQ{1E1} with the boundary condition $\gamma$. It  is enough to show that the $\cS^\prime(\bR^d)$-valued random variable $\tilde u$ defined by
$$
(\tilde u,\psi)= (u,\psi\mid_{\overline{\cO}}),\qquad \psi \in \cS(\bR^d),
$$
solves the problem on the whole space. To do this take $\psi \in \cS(\bR^d)$. Let $\tilde \psi \in \cS(\bR^d)$ be such that
$$
\Delta \tilde \psi =\lambda \tilde \psi \quad \text{in} \quad
\cO,\qquad \tilde \psi = \psi \qquad \text{on}\quad \partial \cO.
$$
By smoothness property of $\tilde \psi$, $\Delta\tilde \psi =\lambda \tilde \psi$ on $\overline{\mathcal{O}}$. Then
$$
\begin{aligned} (\tilde u ,\Delta \psi) &= (u,\Delta (\psi-\tilde \psi)\mid_{\overline{\cO}}) + (u, \Delta \tilde \psi\mid_{\overline{\cO}}) \\
&= (u, \Delta (\psi-\tilde \psi)\mid_{\overline{\cO}})+\lambda (u, \tilde \psi \mid _{\overline{\cO}}) \\
&= \lambda (u,(\psi-\tilde \psi)\mid_{\overline{\cO}})+\lambda (u, \tilde \psi \mid _{\overline{\cO}}) \\
&=\lambda(u,\psi\mid_{\overline {\cO}})= \lambda (\tilde u,\psi).
\end{aligned}
$$
\end{proof}

For the parabolic problem the uniqueness follows from the equivalence of mild and weak formulation of the solution to the homogeneous problem
\begin{equation}\label{3E5}
\d u = \Delta _\tau u \d t,\qquad u(0)=u_0.
\end{equation}

\begin{proposition}\label{3P2}
For any $u_0\in \mathcal{H}^\prime$ and any   $\cS^\prime(\partial \cO)$-valued random process   $\xi$, the parabolic problem $\REQ{1E2}$ has  at most one solution. Moreover, the solutions to $\REQ{3E5}$ and $\REQ{1E2}$ with $\xi\equiv 0$ coincide and are equal to $T(t)u_0$, $t\ge 0$.
\end{proposition}
\begin{remark}\label{1R2}
{\rm Let  $u_\xi$ and $u_\eta$ be two  solutions corresponding to $\REQ{1E2}$ with the same initial value $u_0$ but with boundary functions $\xi$ and $\eta$. Then $u_\xi\equiv u_\eta$ implies that
$$
\left(\xi(t), \frac {\partial \psi}{\partial \mathbf{n}}\bigg \vert_{\partial \mathcal{O}}\right)= \left(\eta(t), \frac {\partial \psi}{\partial \mathbf{n}}\bigg \vert_{\partial \mathcal{O}}\right), \qquad \mathbb{P}-a.s. \ \forall\, t\ge 0, \ \forall \, \psi\in \mathcal{H}.
$$
One can  shown  that there exists  $\psi\in \mathcal{H}$ such that $\frac {\partial \psi}{\partial \mathbf{n}}\not = 0$ on $\partial\mathcal{O}$. }
\end{remark}

\section{Mild formulation and existence}\label{S4}

To solve the elliptic problem we will need  the resolvent operator $\left(\Delta_\tau-\lambda\right)^{-1}$, of the heat semigroup $T$ considered on $L^2(\mathcal{O})$ with $\lambda$ from the resolvent set of $\Delta_\tau$. Then for any $\psi\in \mathcal{S}(\overline{\mathcal{O}})\subset L^2(\mathcal{O})$, the function
$$
u=\left(\Delta_\tau-\lambda\right)^{-1}\psi
$$
solves
$$
(\Delta -\lambda)u(x)=\psi(x), \quad x\in \cO,\qquad u(x)= 0, \quad x\in \partial{\cO}.
$$
By the classical Agmon--Douglis--Nirenberg theory we have $u\in \cS(\overline{\mathcal{O}})$.  Moreover,  $(\Delta_\tau-\lambda)$ is a bounded linear operator from $\cS(\overline{\cO})$ into  itself. Define a continuous  linear operator $D\colon \cS^\prime(\partial \cO)\mapsto \cS^\prime(\overline{\cO})$ by the formula
$$
\left(D\gamma, \psi\right) = \left(\gamma, -\frac{\partial}{\partial\mathbf{n}}\left(\Delta_\tau-\lambda\right)^{-1}\psi\bigg \vert_{\partial \mathcal{O}} \right),\qquad \gamma\in \cS^\prime(\partial \cO),\ \psi\in \cS(\overline{\mathcal{O}}).
$$
We call $D$ the \emph{Dirichlet map}. Using the formula above we extend the Dirichlet map to random variables and obtain the following result.
\begin{theorem}\label{4T1}
For any  $\mathcal{S}^\prime(\partial {\mathcal{O}})$-valued random variable $\gamma$, $D \gamma$ is the unique solution to the Poisson equation $\REQ{1E1}$ in the sense of Definition \ref{3D1}.
\end{theorem}
\begin{proof} Let $\psi\in \mathcal{S}(\overline{\cO})$ be such that $\psi=0$ on $\partial \cO$. Then
\begin{align*}
\left(D\gamma,\Delta\psi\right) &=\left(\gamma, -\frac{\partial }{\partial \mathbf{n}}\left(\Delta _\tau-\lambda \right)^{-1}\Delta_\tau \psi\bigg \vert_{\partial \mathcal{O}}\right)\\
&= \left(\gamma, -\frac{\partial }{\partial {\mathbf{n}}} \psi\bigg \vert_{\partial \mathcal{O}}\right)+ \left(\gamma, -\frac{\partial }{\partial \mathbf{n}}\left(\Delta_\tau-\lambda\right)^{-1} \lambda\psi\bigg \vert_{\partial \mathcal{O}}\right)\\
&= \left(\gamma, -\frac{\partial }{\partial {\mathbf{n}}} \psi\bigg \vert_{\partial \mathcal{O}}\right)+\lambda  \left(D \gamma, \psi\right).
\end{align*}
\end{proof}
The following result ensures the Markov property of solution to $\REQ{1E2}$ in case of $\xi$ having independent increments, see e.g. \cite{DZ1,DZ2, Peszat-Zabczyk}.
\begin{theorem}\label{4T2}
For any $u_0\in \mathcal{H}^\prime$ and any  $\mathcal{S}^\prime(\partial \mathcal{O})$-valued c\`adl\`ag random process $\xi$, the unique solution to $\REQ{1E2}$ is given by formula
\begin{equation}\label{4E6}
u(t)= T(t)u_0 + \int_0^t T(t-s) \left(\lambda-\Delta_\tau\right)D \frac{\d \xi}{\d s}(s),
\end{equation}
where the integrant is defined, on a test function $\psi \in \mathcal{H}^\prime$, by the integration by parts formula
\begin{multline*}
\left( \int_0^t T(t-s) \left(\lambda-\Delta_\tau\right)D\frac{\d \xi}{\d s}(s),\psi\right)\\
= \int_0^t \left(T(t-s) \Delta_\tau \left(\lambda-\Delta_\tau\right)D\xi(s), \psi\right)\d s \\
+ \left(\left(\lambda -\Delta_\tau\right)D \xi(t),\psi\right)-\left(\left(\lambda -\Delta_\tau\right)T(t)D \xi(0),\psi\right).
\end{multline*}
Moreover, $u$ is a Markov and Feller process in $\mathcal H^\star$.
\end{theorem}
\begin{proof}
First note that $u$ is given by the right hand side of \REQ{4E6} is well defined and c\`adl\`ag. To see that it is weak solution take a  $\psi\in \cH$. Write
$$
R(t):= \int_0^t \left((\lambda-\Delta_\tau)D \xi(r) -
(\lambda-\Delta_{\tau}) T(r)D \xi(0), \Delta_\tau \psi\right)\d r.
$$
Then we have
$$
\begin{aligned}
&\int_0^t (u(r),\Delta_\tau\psi)\d r = \left(\int_0^tT(r)u_0\d r, \Delta_\tau \psi\right)\\
&\quad + \left(\int_0^t\int_0^r
T(r-q)(\lambda-\Delta_{\tau})\Delta
_{\tau}D \xi(q)\d q\d r,\Delta _\tau \psi\right) +R(t)\\
&= \left(  \int_0^t \Delta_{\tau}T(t-r)u_0\d
r,\psi\right) \\
&\qquad +\left( \int_0^t \int_q^t T(r-q)(\lambda-\Delta_{\tau})\Delta_{\tau}D  \xi(q)
\d r\d q,\Delta_\tau \psi\right)+R(t)\\
&= \left( -u_0 + T(t)u_0,\psi\right) \\
&\qquad +\left(  \int_0^t \left( T(t-q)
-I\right)\left(\lambda -\Delta_{\tau}\right)D
\xi(q)\d q,\Delta_\tau \psi\right)+R(t)\\
&= -\left(u_0,\psi\right) +\left( u(t),\psi\right) -\int_0^t
\left(\left(\lambda-\Delta_{\tau}\right)D \xi(q),\Delta_\tau
\psi\right) \d q + R(t) +r(t),
\end{aligned}
$$
where
$$
r(t):=-\left( (\lambda-\Delta_{\tau})D \xi(t) - (\lambda-\Delta_{\tau}) T(t)D \xi(0) ,\psi\right).
$$
Now note that
$$
\begin{aligned}
&-\int_0^t \left(\left(\lambda-\Delta_{\tau}\right)D
\xi(q),\Delta_\tau \psi\right) \d q + R(t)\\
&=-\int_0^t \left( (\lambda-\Delta_{\tau})
T(q)D \xi(0), \Delta_\tau \psi\right)\d q\\
&= \left(-T(t)(\lambda -\Delta_{\tau})D\xi(0)+
(\lambda-\Delta_{\tau})D \xi(0),  \psi\right).
\end{aligned}
$$
Therefore
$$
\begin{aligned}
&\int_0^t (u(r),\Delta_\tau\psi)\d r \\
&\qquad =-\left(u_0,\psi\right) +\left( u(t),\psi\right) -
\left(\left[\lambda-\Delta_{\tau}\right]\left[D
\xi(t)-D \xi(0)\right],\psi\right).
\end{aligned}
$$
Since
$$
\begin{aligned}
\left(\left[\lambda-\Delta_{\tau}\right]\left[D
\xi(t)-D \xi(0)\right],\psi\right)= \left(D \xi(t)-D
\xi(0),(\lambda-\Delta)\psi\right),
\end{aligned}
$$
we eventually have
$$
\begin{aligned}
(u(t),\psi) &= (u_0,\psi) + \int_0^t (u(r),\Delta\psi)\d r +
\left(D \xi(t)-D \xi(0),(\lambda-\Delta_\tau)\psi\right).
\end{aligned}
$$
This proves the first part of the theorem since
$$
\left(D\xi(t)-D \xi(0),(\lambda-\Delta_\tau)\psi\right)=
\left(\xi(t)-\xi(0), \frac{\partial \psi}{\partial \mathbf{n}}\bigg \vert_{\partial \mathcal{O}}\right).
$$
The Markov property follows by the same arguments as in \cite{DZ1} and the Feller property is obvious.
\end{proof}

\subsection{Examples}
\begin{example}\label{Ex1}
{\rm In the case of ${\cO}= (0,1)$, the space of distributions $\cS^\prime(\partial \cO)$ can be identified with ${\bR^2}$. Then, taking $\lambda =0$, we obtain for $(\gamma_1,\gamma_2)\in \mathbb{R}^2$,
$$
D(\gamma_0,\gamma_1)(x)= \gamma_0 + (\gamma_1-\gamma_0)x,
\qquad x\in (0,1), \ \gamma =(\gamma_0,\gamma_1)\in \bR^2.
$$
}
\end{example}
\begin{example}\label{Ex2}
{\rm Assume that $\cO= (0,+\infty)$. Then $ \mathcal{S}^\prime(\partial\mathcal{O})\equiv  {\bR}$,  and the Dirichlet map corresponding $\lambda =1$ calculated on $\gamma\in \mathbb{R}$ is   given by
$$
D\gamma(x) = \gamma \e ^{-x},\qquad x\ge 0.
$$
}
\end{example}
\begin{example}\label{Ex3}
{\rm Assume that ${\cO}=\{x\in \bR^d\colon |x|<1\}$ is the unit ball
in $\bR^d$ with center at $0$. Then the Dirichlet map corresponding
to $\lambda =0$ is given on a function $\gamma \colon \partial{\mathcal{O}}\mapsto \mathbb{R}$  by the Poisson integral
$$
D(x) = C_d \int_{\partial {\cO}} \frac{ 1-|x|^2}{|x-y|^{d}}
\gamma(y) \vartheta(\d y),
$$
where $\vartheta$ is the surface measure. }
\end{example}

In the mild formulation of \REQ{1E2} we will deal with the term $\Delta_{\tau}\xi$.  In the result below we calculate $\Delta_{\tau}$ in one dimensional case. To do this let us denote by $\delta_a$ the Dirac measure at $a$ and by $\delta_a^\prime$ its distributional derivative. Recall that $\Delta_\tau$ is a continuous linear operator on $\mathcal{H}^\prime$.
\begin{proposition}
$(i)$ Assume that $\cO=(0,1)$ and $\lambda=0$. Let $\psi_1(x)= 1$
and $\psi_2(x) =x$ for $x\in (0,1)$. Then $\psi_i\in \mathcal{H}^\prime$ and
$$
\Delta_{\tau}\psi_1= \delta_0^\prime-\delta_{1}^\prime \qquad \text{and}\qquad
\Delta_{\tau}\psi_2 =-\delta_{1}^\prime.
$$
$(ii)$ Assume that $\cO=(0,+\infty)$ and $\lambda=1$. Let
$\psi(x)=\e ^{-x}$. Then $\psi\in \mathcal{H}^\prime$ and
$$
\Delta_{\tau}\psi= -\delta_0^\prime + \psi.
$$
\end{proposition}
\begin{proof}
We are showing $(i)$. We have  $\psi_i\in \mathcal{S}^\prime(\overline{\mathcal{O}})\subset \mathcal{H}^\prime$. Let $\phi\in \mathcal{H}$. Recall that  $\Delta_\tau$  also denotes the adjoint operator on $\mathcal{H}$. We have
$$
(\Delta_{\tau}\psi_i,\phi) =(\psi_i, \Delta_{\tau} \phi)=\int_0^1 \psi_i(x) \phi^{\prime\prime}(x)\d x.
$$
Now
$$
\int_0^1 \psi_1(x)\phi^{\prime\prime}(x)\d x=\int_0^1 \phi^{\prime\prime}(x)\d x= \phi^\prime(1)-\phi^\prime(0)= (\delta_{0}^\prime-\delta_{1}^\prime,\phi)
$$
and
$$
\begin{aligned}
\int_0^1\psi_2(x)\phi^{\prime\prime}(x)\d x &= \int_0^1x\phi^{\prime\prime}(x)\d x = \phi^\prime(1)-
\int_0^1\phi^\prime(x)\d x = \phi^\prime(1)\\
&= (-\delta_1^\prime,\phi).
\end{aligned}
$$

In order to show $(ii)$ take a test function $\phi\in \mathcal{H}$. Since $\phi(0)=0$, we have
$$
\int_0^\infty e^{-x} \phi^{\prime\prime}(x)\d x = \phi^\prime(0)+ \int_0^\infty \e ^{-x} \phi^\prime(x) \d x = \phi^\prime(0)+\int_0^\infty \e ^{-x}\phi(x)\d x.
$$
\end{proof}
\section{Green kernels}\label{S5}

For further purposes we recall here some basic facts on Green kernels to the heat and Poisson equations. Assume that $G$ is the Green kernel to the heat equation in $\mathcal O\subset\mathbb R^d$ with homogeneous Dirichlet boundary conditions. Thus the semigroup $T$ generated by $\Delta_\tau$ is given by
$$
T(t)\psi(x)= \int_{\cO} G(t,x,y) \psi(y)\d y.
$$
Note that as $\Delta_\tau$ is self-adjoint, $G(t,x,y)=G(t,y, x)$. Moreover, as the boundary and the coefficients of the operator $\Delta$ are $C^\infty$, we have the following bounds on $G$ and its derivatives (see \cite{Eidelman-Ivasishen, Solonnikov, Mora}).
\begin{theorem}\label{5T3}
The fundamental solution $G$ is of class $C^\infty ((0,+\infty)\times \cO\times \cO)$ and for any non-negative integer $n$, a multi-index $\alpha $, and time $S>0$, there are constants $K_1,K_2>0$ such that for all $t\in (0,S]$ and $z\in \mathcal{O}$,
$$
\left \vert \frac{\partial^n}{\partial t^n} \frac{\partial ^{|\alpha|}}{\partial z^\alpha } G(t,x,y)\right\vert \le K_1 t ^{-(d+|\alpha|+ 2n)/2} \e^{-\frac{|x-y|^2}{K_2t}}.
$$
\end{theorem}

On the space $C^\infty((0,+\infty)\times \mathcal{O})$ we consider a family of semi-norms
$$
p_{K,S}^{n,m}(f) := \sup_{x\in K} \sup_{t\in (0,S]} \sup_{\alpha\colon |\alpha|\le n} \sup_{k=0,\ldots,m} \left\vert \frac{\partial ^{k+|\alpha|} f}{\partial t^k \partial x^\alpha} (t,x)\right\vert,
$$
where $n,m\in \mathbb{N}$ and $K$ is a compact subset of $\mathcal{O}$.

The corollaries below can be easily derived from  Theorem \ref{5T3}.
\begin{corollary}
For any compact $K\subset \cO$, any finite time $S<\infty$, and all $n,m\in \mathbb{N}$,
$$
\sup_{y\in \partial \cO} p_{K,S}^{n,m}\left(G(\cdot,\cdot,y)\right)<\infty.
$$
Moreover, for any $x\in \mathcal{O}$, $\delta>0$ such that the ball $B(x,\delta) \subset \mathcal{O}$, and for any $t>0$,
$$
\frac{\partial G}{\partial \mathbf{n}_y}(t, x,\cdot )\in \mathcal{S}(\overline{\mathcal{O}\setminus B(x,\delta)}).
$$
\end{corollary}

Let
$$
\cG(x,y)=\int_0^{+\infty} \e ^{-\lambda t} G(t,x,y)\d t
$$
be the Green kernel to the Poisson equation.  On the space $C^\infty(\mathcal{O})$ we consider a family of semi-norms
$$
p_{K}^{n}(f) := \sup_{x\in K} \sup_{\alpha\colon |\alpha|\le n} \left\vert \frac{\partial ^{|\alpha|} f}{\partial x^\alpha} (x)\right\vert,
$$
where $n\in  \mathbb{N}$ and $K$ is a compact subset of $\mathcal{O}$.

\begin{corollary}\label{5C2}
There is a constant $C$ such that for all $x\in \cO$ and $y\in \partial \cO$,
$$
\left\vert \frac{\partial \cG}{\partial \mathbf{n}_y}(x,y)\right\vert \le
\begin{cases}
C \, |x-y|^{1-d} & \text{\rm if $d>1$}, \\
C\left[1+\log^+\, |x-y|^{-1} \right] &\text{\rm if d=1}.
\end{cases}
$$
Moreover, or any $y\in \partial \cO$, $\frac{\partial \cG}{\partial \mathbf{n}}(\cdot,y)\in C^\infty(\cO)$ and for any compact $K\subset \cO$ and $n$ there is a constant $C$ such that
$$
p_K^n \left( \frac{\partial \cG}{\partial \mathbf{n}_y}(\cdot,y)\right)\le C \qquad \forall\, y\in \partial \cO.
$$
Finally  for any $x\in \mathcal{O}$ and  $\delta>0$ such that  $B(x,\delta) \subset \mathcal{O}$,
$$
\frac{\partial \cG}{\partial \mathbf{n}_y}(x,\cdot )\in \mathcal{S}(\overline{\mathcal{O}\setminus B(x,\delta)}).
$$
\end{corollary}

Let $\gamma \in \mathcal{S}^\prime(\partial \mathcal{O})$. Write
\begin{equation}\label{5E7}
u_\gamma(x):= \left( \gamma,  \frac{\partial \cG}{\partial \mathbf{n}_y}(x,\cdot)\bigg \vert_{\partial \mathcal{O}}\right), \qquad x\in \mathcal{O}
\end{equation}
and
\begin{equation}\label{5E8}
v_\gamma(t,x):= \left(\gamma , \frac{\partial G}{\partial \mathbf{n}_y}(t,x,\cdot)\bigg \vert_{\partial \mathcal{O}}\right),\qquad t>0,\ x\in \mathcal{O}.
\end{equation}
\begin{corollary}\label{5C3}
For any $\gamma \in \mathcal{S}^\prime(\partial \mathcal{O})$, $u_\gamma\in C^\infty(\mathcal{O})$ and $v_\gamma\in C^\infty([0,+\infty)\times \mathcal{O})$.
\end{corollary}

\begin{example}\label{5Ex1}
{\rm Assume that $\mathcal{O}=(0,+\infty)\times \mathbb{R}^m$. Then
$G(t,x,y)= \Gamma (t,x-y)-\Gamma (t,\overline x-y)$, where
$$
\Gamma(t,x)= \left(4\pi t\right)^{-(m+1)/2}\e^{-\frac{|x|^2}{4t}},
$$
and
$$
\overline {x}=\overline{(x_0,x_1,\ldots,x_m)}=(-x_0,x_1,\ldots,x_m).
$$
Then
$$
\frac{\partial G}{\partial \mathbf{n}_y}(t,x,y)= -\frac{\partial G}{\partial y_0}(t,x,y)= \frac{y_0-x_0}{2t}\Gamma(t,x-y)+ \frac{-y_0-x_0}{2t}\Gamma(t,\overline x-y).
$$
In particular we have
$$
\frac{\partial G}{\partial \mathbf{n}_y}(t,x,y)= -\frac{x_0}t\Gamma(t,x-y), \qquad y=(0,y_1,\ldots, y_m)\in \partial \mathcal{O}.
$$
Let  $y=(0,y_1,\ldots,y_m)\in \partial \mathcal{O}$. We will need to calculate the Fourier transform
\begin{align*}
\mathcal{F}_y\frac{\partial G}{\partial \mathbf{n}_y}(t,x,y)&=-\frac{x_0}t\int_{\mathbb{R}^m}\e^{\i \langle (z_1,\ldots,z_m),(y_1,\ldots,y_m)\rangle} \Gamma(t,x-(0,z_1,\ldots,z_m)) \d x_1\ldots \d z_m\\
&= -\frac{x_0}{2\sqrt{\pi}t^{3/2}} \e^{-\frac{x_0^2}{4t}}\e^{\i \langle x,y\rangle - \frac{(2t)^m}{2}|y|^2}.
\end{align*}
Finally, for
$\mathcal{G}(x,y) = \int_0^{+\infty} \e ^{-\lambda t}G(t,x,y)\d t$, we have at $y=(0,y_1,\ldots,y_m)$,
$$
\frac{\partial \mathcal{G}}{\partial \mathbf{n}_y}(x,y) = -\frac{x_0}{t} \int_0^{+\infty} \e^{-\lambda t} \Gamma(t,x-y)\d t
$$
and
\begin{align*}
\mathcal{F}_y\frac{\partial \mathcal{G}}{\partial \mathbf{n}_y}(x,y)&= -\int_0^{+\infty}\e^{-\lambda t}\frac{x_0}{2\sqrt{\pi}t^{3/2}}\e^{-\frac{x_0^2}{4t}}\e^{\i \langle x,y\rangle - \frac{(2t)^m}{2}|y|^2}\d t.
\end{align*}
In the calculation below $c$ is a generic constant.  For $|y|\ge 1$ we have
\begin{align*}
\left\vert \mathcal{F}_y\frac{\partial \mathcal{G}}{\partial \mathbf{n}_y}(x,y)\right\vert^2  &\le c \int_0^{+\infty} \e^{-\lambda t} \frac{x_0^2}{t^3}\e^{-\frac{x_0^2}{2t}-(2t)^m|y|^2}\d t \\
&\le c x_0^{-2} \int_0^{+\infty} \e^{-\lambda x_0^2 s-2^m s^m x_0^{2m}|y|^2}\d s \\
&\le c x_0^{-3} \left(\int_0^{+\infty} \e^{-2^{m+1} s^m x_0^{2m}|y|^2}\d s \right)^{1/2}\\
&\le cx_0^{-4}|y|^{-1/m}.
\end{align*}
Since for $|y|\le 1$,
$$
\left\vert \mathcal{F}_y\frac{\partial \mathcal{G}}{\partial \mathbf{n}_y}(x,y)\right\vert^2  \le c \int_0^{+\infty} \e^{-\lambda t} \frac{x_0^2}{t^3}\e^{-\frac{x_0^2}{2t}}\d t
\le c x_0^{-2},
$$
eventually we obtain  the following a bit crude estimate
\begin{equation}\label{5E9}
\left\vert \mathcal{F}_y\frac{\partial \mathcal{G}}{\partial \mathbf{n}_y}(x,y)\right\vert^2  \le c \left( x_0^{-2}\chi_{\{|y|\le 1\}} + x_0^{-4} |y|^{-1/m}\chi_{\{|y|> 1\}}\right).
\end{equation}
In the same way on obtains
\begin{equation}\label{5E10}
\int_0^t \left\vert \mathcal{F}_y\frac{\partial {G}}{\partial \mathbf{n}_y}(s,x,y)\right\vert^2 \d s \le c \left( x_0^{-2}\chi_{\{|y|\le 1\}} + x_0^{-4} |y|^{-1/m}\chi_{\{|y|> 1\}}\right)\e^{t}.
\end{equation}
}
\end{example}

\section{Regularity of solutions inside  domain}\label{S6}

In this section we are concerned with the regularity of solutions inside  the domain $\mathcal{O}$. Let us denote by  $C_c^\infty(\mathcal{O})$ the space of all compactly supported $C^\infty$ functions on $\mathcal{O}$. Recall that $u_\gamma\in C^\infty(\mathcal{O})$ and $v_\gamma\in C^\infty((0,+\infty)\times\mathcal{O})$ are defined by \REQ{5E7} and \REQ{5E8}, respectively.
\begin{theorem}\label{6T4}
$(i)$ Let $\gamma$ be an $\mathcal{S}^\prime(\partial \mathcal{O})$-valued random variable, and let $u=D\gamma $ be the solution to the elliptic problem $ \REQ{1E1}$. Then $u=u_\gamma$ on $\mathcal{O}$; that is
$$
\left(u,\psi\right)= \int_{\mathcal{O}}u_\gamma(x)\psi(x)\d x
$$
for any  $\psi\in C_c^\infty(\mathcal{O})$.
Moreover, the solution $u$ is $C^\infty$ inside the domain $\mathcal{O}$.
\noindent
$(ii)$ Let $\xi$ be a process with c\`adl\`ag trajectories in $\mathcal{S}^\prime(\partial \mathcal{O})$ and let $u$ be the solution to the parabolic problem $\REQ{1E2}$. Then
$$
\left(u, \psi\right) = \int_\mathcal{O} u_{\xi, u_0}(t,x)\psi(x)\d x, \qquad \psi\in C^\infty_c(\mathcal{O}),
$$
where
\begin{equation}\label{6E9}
u_{\xi,u_0}(t,x)= \int_{\mathcal{O}}G(t,x,y)u_0(y)\d y + \int_0^t \Delta_x v_{\xi(s)}(t-s,x)\d s - v_{\xi(0)}(t,x),
\end{equation}
and in particular,   $u$ is  $C^\infty$ on $(0,+\infty)\times \mathcal{O}$.
\end{theorem}
\begin{proof}  The first part of the theorem says that
$$
\left(D\gamma,\psi\right)=(u_\gamma,\psi),\qquad \forall\, \psi \in C_c^\infty(\mathcal{O}).
$$
We have
\begin{align*}
(D\gamma,\psi) &= \left(\gamma, -\frac{\partial}{\partial\mathbf{n}}\left(\Delta_\tau-\lambda\right)^{-1}\psi\bigg \vert_{\partial \mathcal{O}} \right)= \left(\gamma, \frac{\partial}{\partial\mathbf{n}}\int_{\mathcal{O}}\mathcal{G}(x,\cdot)\psi (x)\d x\bigg \vert_{\partial \mathcal{O}} \right)\\
&=  \left(\gamma, \int_{\mathcal{O}}\frac{\partial \mathcal{G}}{\partial\mathbf{n}_y}(x,\cdot)\psi (x)\d x\bigg \vert_{\partial \mathcal{O}} \right).
\end{align*}
Since $\psi$ has a compact support in $\mathcal{O}$, there is a sequences $(x_k^n)$, $k=1,\ldots,n$,  of variables of the support of $\psi$ and reals $(a_k^n)$ such that
$$
\int_{\mathcal{O}}\frac{\partial \mathcal{G}}{\partial\mathbf{n}_y}(x,\cdot)\psi (x)\d x= \lim_{n\to \infty} \sum_{k=1}^n \overline {\frac{\partial \mathcal{G}}{\partial\mathbf{n}_y}(x_k^n,\cdot)}\psi (x_k^n)a_k^n,
$$
where $\overline {\frac{\partial \mathcal{G}}{\partial\mathbf{n}_y}(x_k^n,\cdot)}\in \mathcal{S}(\mathbb{R}^d)$ is an extension of ${\frac{\partial \mathcal{G}}{\partial\mathbf{n}_y}(x_k^n,\cdot)}$, and the convergence is in the topology of $\mathcal{S}(\mathbb{R}^d)$, and
$$
\lim_{n\to \infty}\sum_{k=1}^n u_\gamma(x^n_k)\psi (x_k^n)a_k^n=\int_{\mathcal{O}}u_\gamma(x)\psi(x)\d x.
$$
Then
\begin{align*}
\left(\gamma, \int_{\mathcal{O}}\frac{\partial \mathcal{G}}{\partial\mathbf{n}_y}(x,\cdot)\psi (x)\d x\bigg \vert_{\partial \mathcal{O}} \right)
&= \lim_{n\to \infty} \sum_{k=1}^n \left(\gamma, \overline {\frac{\partial \mathcal{G}}{\partial\mathbf{n}_y}(x_k^n,\cdot)}\bigg \vert_{\partial \mathcal{O}} \right)\psi (x_k^n)a_k^n\\
&=
\lim_{n\to \infty} \sum_{k=1}^nu_\gamma(x^n_k)\psi (x_k^n)a_k^n=\int_{\mathcal{O}}u_\gamma(x)\psi(x)\d x.
\end{align*}
Regularity of $u$ inside the domain now follows from Corollary \ref{5C3}.
\par\noindent
Let $u$ be the solution to $\REQ{1E2}$. To prove the second part of the theorem, note that
\begin{align*}
\left((\lambda-\Delta_\tau)T(t-s)D\xi(s), \psi\right) &= \left(\xi(s),-\frac{\partial }{\partial\mathbf{n}} (\Delta_\tau-\lambda) ^{-1}(\lambda-\Delta_\tau)T(t-s)\psi\bigg \vert_{\partial \mathcal{O}}\right)\\
&= \left(\xi(s),\frac{\partial }{\partial\mathbf{n}} T(t-s)\psi\bigg \vert_{\partial \mathcal{O}}\right).
\end{align*}
Using the arguments from the proof of the first part we obtain
$$
\left((\lambda-\Delta_\tau)T(t-s)D\xi(s), \psi\right)= \int_{\mathcal{O}}v_{\xi(s)}(t-s,x)\psi(x)\d x, \qquad \forall\, \psi\in C^\infty_c(\mathcal{O})
$$
and
$$
\left(\Delta_\tau(\lambda-\Delta_\tau)T(t-s)D\xi(s), \psi\right)= \int_{\mathcal{O}}(\Delta_xv_{\xi(s)}(t-s,x))\psi(x)\d x, \qquad \forall\, \psi\in C^\infty_c(\mathcal{O}).
$$
Let $\psi\in C^\infty_c(\mathcal{O})$. Taking into account $\REQ{4E6}$ we obtain
\begin{align*}
\left(u(t),\psi\right) &= \left(T(t)u_0,\psi\right) + \int_0^t  \int_{\mathcal{O}}(\Delta_xv_{\xi(s)}(t-s,x))\psi(x)\d x \d s \\
&\qquad + \left(\xi(t), \frac{\partial \psi}{\partial \mathbf{n}}\bigg \vert_{\partial \mathcal{O}}\right)-  \int_{\mathcal{O}}v_{\xi(0)}(t,x)\psi(x)\d x.
\end{align*}
Since $\frac{\partial \psi}{\partial \mathbf{n}}\bigg \vert_{\partial \mathcal{O}}\equiv 0$, we obtain \REQ{6E9}.
\end{proof}

\section{Regularity of solutions at the boundary} \label{S7}

Our  aim in this section is to investigate the space regularity of $u$ defined by \REQ{1E1} or \REQ{1E2} at the boundary.  To this end we will evaluate the expectations $\mathbb{E}u^2(x)$ and $\mathbb{E}u^2(t,x)$ where $x$ is near the boundary $\partial \mathcal{O}$ and we will  obtain the bounds of the type
$$
\mathbb{E}u^2(x)\le f({\rm dist}\, (x,\partial \mathcal{O}))\qquad \text{and}\qquad \mathbb{E}u^2(t,x)\le f({\rm dist}\, (x,\partial \mathcal{O})).
$$
In a similar way one can evaluate $u(x)$ and $u(t,x)$ in order to obtain the estimates of the form
$$
u^2(x;\omega)\le c(\omega)f({\rm dist}\, (x,\partial \mathcal{O}))\qquad \text{and}\qquad u^2(t,x;\omega)\le c(\omega)f({\rm dist}\, (x,\partial \mathcal{O})),
$$
 for $\mathbb{P}$ almost all $\omega$.
\subsection{Poisson equation}
Recall, see Theorem \ref{6T4}, that the solution $u$ to \REQ{1E1} inside of $\mathcal{O}$ can be identified with smooth random field
$$
u(x)= \left( \gamma,  \frac{\partial \cG}{\partial \mathbf{n}_y}(x,\cdot)\bigg \vert_{\partial \mathcal{O}}\right), \qquad x\in \mathcal{O}.
$$

Assume now that $\gamma=\sum_k \gamma_ke_k\nu$ is as in Example \ref{2Ex1}. Then
$$
u(x)= \sum_k\gamma_k \int_{\partial \mathcal{O}}\frac{\partial \cG}{\partial \mathbf{n}_y}(x,y)e_k(y)\nu(\d y),\qquad x\in \mathcal{O}.
$$
Therefore, for every $ x\in \mathcal{O}$,
$$
\mathbb{E}u^2(x)= \sum_k \left(\int_{\partial \mathcal{O}}\frac{\partial \cG}{\partial \mathbf{n}_y}(x,y)e_k(y)\nu(\d y)\right)^2=
\int_{\partial \mathcal{O}}\left(\frac{\partial \cG}{\partial \mathbf{n}_y}(x,y)\right)^2\nu(\d y).
$$
Taking into account the estimates of Corollary \ref{5C2} we obtain the following result.
\begin{proposition}
Let $u$ be the solution to \REQ{1E1} with $\gamma$ as in Example \ref{2Ex1}.  Then
$$
\mathbb{E}u^2(x)\le C\times
\begin{cases}
\int_{\partial\mathcal{O}}|x-y|^{2-2d} \nu(\d y)&\text{if $d>1$,}\\
\left[1+\log^+{\rm dist}(x, \partial \mathcal{O})\right]^2&\text{if $d=1$}.
\end{cases}, \qquad x\in \mathcal{O}.
$$
\end{proposition}

Assume now that $\gamma = \sum_k \gamma_k\nu_k$ is as in Example \ref{2Ex3}. Then, for every $ x\in \mathcal{O}$,
$$
\mathbb{E}u^2(x) = \sum_k \left(\int_{\partial \mathcal{O}}\frac{\partial \cG}{\partial \mathbf{n}_y}(x,y)\nu_k(\d y)\right)^2\le \sum_k\|\nu_k\|^2_{\rm Var} \sup_{y\in \partial \mathcal{O}} \left(\frac{\partial \cG}{\partial \mathbf{n}_y}(x,y)\right)^2,
$$
and consequently we have the following.
\begin{proposition}
Let $u$ be the solution to \REQ{1E1} with $\gamma$ as in Example \ref{2Ex3}. Then
$$
\mathbb{E}u^2(x)\le  C\times
\begin{cases}
 {\rm dist}(x,\partial \mathcal{O})^{2-2d}&\text{if $d>1$,}\\
 \left[1+\log^+{\rm dist}(x, \partial \mathcal{O})\right]^2&\text{if $d=1$},
\end{cases},
\qquad x\in \mathcal{O}.
$$
\end{proposition}

Assume that $\mathcal{O}= (0,+\infty)\times \mathbb{R}^m$. An important case can be obtained if $\gamma=\mathcal{W}(1)$ where $\mathcal{W}$ is a homogeneous Wiener process, see Example \ref{2Ex4}. Then $\gamma =\sum_k W_k(1)\mathcal{F}(e_k\nu)$. Obviously
$$
u(x_0,x_1,\ldots,x_m)=\sum_k W_k(1) \int_{\mathbb{R}^m} \frac{\partial \cG}{\partial \mathbf{n}_y}((x_0,\ldots x_m, 0, y_1,\ldots y_m)\mathcal{F}(e_k\nu)(y)\d y.
$$
Consequently, see Example \ref{5Ex1}, for every $ x\in \mathcal{O}$,
\begin{align*}
\mathbb{E}u^2(x)&= \int_{\mathbb{R}^m} \left\vert \mathcal{F}_y\frac{\partial \cG}{\partial \mathbf{n}_y}((x_0,\ldots x_m, 0, y_1,\ldots y_m)\right\vert ^2 \nu(\d y)\\
&=
\int_{\mathbb{R}^m} \left\vert \int_0^{+\infty}\e^{-\lambda t}\frac{x_0}{2\sqrt{\pi}t^{3/2}}\e^{-\frac{x_0^2}{4t}}\e^{\i \langle x,y\rangle - \frac{(2t)^m}{2}|y|^2}\d t\right\vert ^2 \nu(\d y).
\end{align*}
Using now the crude estimate \REQ{5E9} we obtain:
\begin{proposition}
Let $u$ be the solution to \REQ{1E1} with $\gamma$ as in Example \ref{2Ex4}. Then
$$
\mathbb{E}u^2(x)\le C\int_{\mathbb{R}^m} \left( x_0^{-2}\chi_{\{|y|\le 1\}} + x_0^{-4} |y|^{-1/m}\chi_{\{|y|> 1\}}\right)
\nu(\d y),\qquad x\in \mathcal{O}.
$$
\end{proposition}
\begin{remark}
{\rm Note that for an arbitrary (tempered) spectral measure $\nu$, $u$ is a Gaussian field in $\mathcal{O}$ and consequently $\mathbb{E}u^2(x)<+\infty$ even if the integral appearing on the right hand side of the estimate in the proposition above is infinite.
}
\end{remark}
Let us now consider the case of L\'evy measure on the boundary. Namely, see Example \ref{2Ex5},  assume that $\gamma=\rho \pi$, where $\rho$ is a function and $\pi$ is a Poisson random measure on $\partial \mathcal{O}$ with intensity measure $\nu$. Then
$$
u(x) = \int_{\partial \mathcal{O}}\frac{\partial \cG}{\partial \mathbf{n}_y}(x,y)\rho(y) \pi(\d y)\qquad x\in \mathcal{O}.
$$
Taking into account, see e.g. \cite{Peszat-Zabczyk}, the following estimate valid for any measurable $f\colon E\mapsto \mathbb{R}$ and and Poisson random measure $\pi$ on $E$ with intensity measure $\nu$,
$$
\mathbb{E}\left(\int_E f(x)\pi(\d x)\right)^2 \le 2\left[\int_Ef^2(x)\nu(\d x)+ \left(\int_{E}f(x)\nu(\d x)\right)^2\right], \qquad x\in \mathcal{O},
$$
and taking into account the estimates of Corollary \ref{5C2} we obtain the following result.
\begin{proposition}
Let $u$ be the solution to \REQ{1E1} with $\gamma$ as in Example \ref{2Ex5}.  Then, for $ x\in \mathcal{O}$
$$
\mathbb{E}u^2(x)\le C\times
\begin{cases}
\int_{\partial\mathcal{O}}|x-y|^{2-2d} \rho(y)^2\nu(\d y)+ \left(\int_{\partial\mathcal{O}}|x-y|^{1-d} |\rho(y)|\nu(\d y)\right)^2&\text{if $d>1$,}\\
\left[1+\log^+{\rm dist}(x, \partial \mathcal{O})\right]^2&\text{if $d=1$}.
\end{cases}
$$
\end{proposition}
\begin{remark}
{\rm In the L\'evy case with a bounded $\rho$ and finite $\nu$, it is easy to obtain pointwise estimate
$$
|u(x;\omega)| \le C(\omega)\times
\begin{cases}
\left({\rm dist}\, x,\partial \mathcal{O}\right)^{1-d} &\text{if $d>1$,}\\
\left[1+\log^+{\rm dist}(x, \partial \mathcal{O})\right]&\text{if $d=1$},
\end{cases}
\qquad x\in \mathcal{O}.
$$
where the random variable $C$ has all moments finite. Using the fact that for $p\in [1,2)$
$$
\mathbb{E}\left\vert \int_Ef(y)\pi(\d y)\right\vert ^p\le C_p \int_{E}|f(y)|^p\nu(\d y)
$$
one can show that
$$
\mathbb{E}|u(x;\omega)|^p \le C\times
\begin{cases}
\int_{\partial\mathcal{O}}|x-y|^{p-pd} |\rho(y)|^p\nu(\d y)&\text{if $d>1$,}\\
\left[1+\log^+{\rm dist}(x, \partial \mathcal{O})\right]^p&\text{if $d=1$}
\end{cases}\qquad x\in \mathcal{O}.
$$
For more details see \cite{Peszat-Zabczyk}.
}
\end{remark}
\subsection{Parabolic case} Let us now examine the case of the heat problem. Without any loss of generality we may assume that $u_0=0$ and $\xi(0)=0$. Then, see Theorem \ref{6T4},
$$
u(t,x) =u_{\xi,0}(t,x) = \int_0^t \Delta_x v_{\xi(s)}(t-s,x)\d s= \int_0^t \left(\xi(s), \Delta_x\frac{\partial G}{\partial \mathbf{n}_y}(t-s,x,\cdot)\right)\d s.
$$
Assume that $\xi (t)=\sum_kW_k^H(t)e_k\nu$ is as in Example \ref{2Ex2}. Then
\begin{align*}
u(t,x) &= \int_0^t \left(\d \xi(s), \frac{\partial G}{\partial \mathbf{n}_y}(t-s,x,\cdot)\right)
= \sum_k \int_0^t \left(e_k\nu, \frac{\partial G}{\partial \mathbf{n}_y}(t-s,x,\cdot)\right)\d W_k^H(s),
\end{align*}
where the integral is in the It\^o sense. Therefore, for every $ x\in \mathcal{O}$,
$$
\mathbb{E}\, |u(t,x)|^2 =\sum_k \mathbb{E}\left\vert \int_0^t \left(\int_{\partial \mathcal{O}}\frac{\partial G}{\partial \mathbf{n}_y}(t-s,x,y)e_k(y)\nu(\d y)\right)\d W_k^H(s)\right\vert^2.
$$
Now if $H=1/2$, then
$$
\mathbb{E}\, |u(t,x)|^2 = \int_0^t \int_{\partial \mathcal{O}}\left(\frac{\partial G}{\partial \mathbf{n}_y}(s,x,y)\right)^2\nu(\d y)\d s.
$$
Using now Theorem \ref{5T3}, we obtain
$$
\mathbb{E}\, |u(t,x)|^2 \le K_1 \int_0^t s^{-d-1} \e^{-\frac{|x-y|^2}{2K_2s}}\nu(\d y)\d s.
$$
Since
\begin{equation}\label{7E1}
\int_0^t s^{-d-1} \e^{-\frac{|x-y|^2}{2K_2s}}\d s\le C|x-y|^{-2d},
\end{equation}
we have the following result.
\begin{proposition}
Assume that $\xi=\sum_k W_ke_k\nu $ is as in Example \ref{2Ex2} with $H=1/2$. Then
$$
\mathbb{E}\, |u(t,x)|^2 \le C\int_{\partial \mathcal{O}}|x-y|^{-2d}\nu(\d y).
$$
\end{proposition}

Assume now that $\xi=\sum_k W_k^H\nu_k$ is as in Example \ref{2Ex3}. Then
$$
\mathbb{E}\, |u(t,x)|^2 =\sum_k \mathbb{E}\left\vert \int_0^t \left(\int_{\partial \mathcal{O}}\frac{\partial G}{\partial \mathbf{n}_y}(t-s,x,y)\nu_k(\d y)\right)\d W_k^H(s)\right\vert^2.
$$
Now if $H=1/2$, then
\begin{align*}
\mathbb{E}\, |u(t,x)|^2 &= \sum_k \int_0^t \left(\int_{\partial \mathcal{O}}\frac{\partial G}{\partial \mathbf{n}_y}(s,x,y)\nu_k(\d y)\right)^2\d s\\
&\le \int_0^t \sup_{y\in \partial \mathcal{O}}\left\vert \frac{\partial G}{\partial \mathbf{n}_y}(s,x,y)\right\vert^2\sum_k \|\nu_k\|^2_{\rm Var}.
\end{align*}
Therefore the following result holds.
\begin{proposition}
Assume that $\xi=\sum_k W_ke_k\nu $ is as in Example \ref{2Ex3} with $H=1/2$. Then
$$
\mathbb{E}\, |u(t,x)|^2 \le C{\rm dist}\, (x,\partial \mathcal{O})^{-2d}.
$$
\end{proposition}

Assume now that $\mathcal{O}=(0,+\infty)\times \mathbb{R}^m$ and that $\xi=\mathcal{W}$ is a spatially homogeneous Wiener process on $\mathbb{R}^m$ with the spectral measure $\nu$. Then $\xi=\sum_kW_k\mathcal{F}(e_k\nu)$. Hence, see Example \ref{5Ex1},
\begin{align*}
\mathbb{E}\, |u(t,x)|^2 &= \sum_k \int_0^t \left(\int_{\mathbb{R}^m}\frac{\partial G}{\partial \mathbf{n}_y}(s,x,0,y)\mathcal{F}(e_k\nu)(y)\d y\right)^2\d s\\
&= \int_0^t \int_{\mathbb{R}^m} \left\vert \mathcal{F}_y \frac{\partial G}{\partial \mathbf{n}_y}(s,x,0,y)\right\vert ^2 \nu(\d y)\d s\\
&= \int_0^t \frac{x^2_0}{4\pi s^3} \e^{-\frac{x_0^2}{2s}}\int_{\mathbb{R}^m}\e^{-(2s)^m |y|^2}\nu(\d y)\d s.
 \end{align*}
Therefore, using \REQ{5E10} we obtain:
\begin{proposition}
Assume that $\xi$ is a spatially homogeneous Wiener process as  in Example \ref{2Ex4} with the spectral measure $\nu$. Then
$$
\mathbb{E}\, |u(t,x)|^2 \le C\e^t \int_{\mathbb{R}^m} \left( x_0^{-2}\chi_{\{|y|\le 1\}} + x_0^{-4} |y|^{-1/m}\chi_{\{|y|> 1\}}\right) \nu(\d x).
$$
\end{proposition}

Assume that $\xi$ is as in Example \ref{2Ex5}. Then
$$
u(t,x) = \int_0^t \int_{\partial \mathcal{O}}\frac{\partial G}{\partial \mathbf{n}_y}(t-s,x,y)\rho(y)\Pi(\d s,\d y).
$$
Thus
\begin{multline*}
\mathbb{E}\, u^2(t,x) \\
\le 2\left(  \int_0^t \int_{\partial \mathcal{O}}\left\vert \frac{\partial G}{\partial \mathbf{n}_y}(s,x,y)\right\vert ^2 |\rho(y)|^2 \d s\nu(\d y)+ \left( \int_0^t \int_{\partial \mathcal{O}}\left\vert \frac{\partial G}{\partial \mathbf{n}_y}(s,x,y)\right\vert |\rho(y)| \d s\nu(\d y)\right)^2\right)\\
\le C\left( \int_{\partial \mathcal O}\int_0^t s^{-d-1} \e^{-\frac{|x-y|^2}{2Ks}}\d s |\rho(y)|^2\nu(\d y)
+ \left( \int_{\partial \mathcal O}\int_0^t s^{-\frac{d+1}{2}} \e^{-\frac{|x-y|^2}{Ks}}\d s |\rho(y)|\nu(\d y)
\right)^2\right).
\end{multline*}
Taking into account \REQ{7E1} we obtain:
\begin{proposition}
Assume that $\xi$ is as in Example \ref{2Ex5}. Let $d>1$. Then
$$
\mathbb{E}\, u^2(t,x) \le C \left( \int_{\partial \mathcal{O}}|x-y|^{-2d} |\rho(y)|^2\nu(\d y) + \left( \int_{\partial \mathcal{O}}|x-y|^{-d-1} |\rho(y)|\nu(\d y)\right)^2\right).
$$
\end{proposition}
\subsection{The case of $H\not =1/2$}
We restrict our attention to the case of $H>1/2$ and boundary noise  $\xi=\sum_k W_k^H\nu_k $ is as in Example \ref{2Ex3}.  The cases of $H<1/2$ or and $\xi=\sum_{k} W_k^He_k\nu$ require much longer calculations.
\begin{proposition}
Let $H>1/2$. Assume that $\xi=\sum_k W_k^H\nu_k $ is as in  Example \ref{2Ex3}. Then for any $1/2>\alpha>1-H$, $T>0$, and for any $x\in \cO$ there is a random variable $C$ having finite all moments such that
\begin{equation}\label{6E9bis}
\sup_{t\le T} |u(t,x)|\le C \left({\rm dist}\, (x, \partial \mathcal{O})\right)^{-d+1-\alpha}.
\end{equation}
\end{proposition}

The proposition follows  from the estimates of Theorem \ref{5T3}, inequality \REQ{7E1},  and the  estimates established in \cite{Nualart-Rascanu}, Proposition 4.1, valid for a deterministic continuous  $f$ and fractional real-valued Brownian Motion $W^H$;
$$
\left\vert \int_0^t f(s)\d W^H(s)\right\vert \le
\Lambda_\alpha(W^H)I_\alpha(f),
$$
where $I_\alpha(f)$ is given by
$$
I_\alpha(t)(t) :=\int_0^t \left( \frac{|f(r)|}{r^\alpha} +\alpha \int_0^r \frac{|f(r)-f(q)|}{|r-q|^{\alpha +1}}\d q \right)\d r,
$$
$$
\Delta_\alpha (g): = \frac{1}{\Gamma (1-\alpha)} \sup_{0<s<t<T} \left\vert (D^{1-\alpha}_{t-}W^H_{t-})(s)\right\vert
$$
and
$$
D^\alpha _{a+} h(t) :=\frac{(-1)^\alpha}{\Gamma (1-\alpha)} \left(\frac{h(t)}{(a-t)^\alpha} + \alpha \int_t^a \frac{h(t)-h(r)}{(r-t)^{\alpha +1}}\d r \right) \chi_{(0,a)}(t).
$$


\end{document}